\documentclass[review]{elsarticle}

\usepackage{lineno,hyperref}
\modulolinenumbers[5]
\usepackage{graphicx}
\usepackage{mathptmx} 
\usepackage{latexsym}
\usepackage{amsfonts,epsfig}
\usepackage{hyperref}
\usepackage{graphics}
\usepackage{epsfig}
\usepackage{amssymb,amsmath,mathrsfs}
\usepackage{amsthm}
\usepackage{amsfonts}
\usepackage{color}
\usepackage{mathtools}
\usepackage{algorithmic}
\usepackage{algorithm}

\newcommand{\Span}{\mathop{\rm span}\nolimits}

\modulolinenumbers[5]
\newtheorem{theorem}{Theorem}[section]
\newtheorem{lemma}[theorem]{Lemma}

\newtheorem{example}[theorem]{Example}
\newtheorem{corollary}[theorem]{Corollary}

\newtheorem{remark}[theorem]{Remark}

\numberwithin{equation}{section}
\journal{Journal of \LaTeX\ Templates}

\begin{document}

\begin{frontmatter}

\title{A generalized projection iterative methods for solving non-singular linear systems}

\author{Ashif Mustafa} 
\author{Manideepa Saha \fnref{myfootnote}\corref{mycorrespondingauthor}}
\address{Department of Mathematics, National Institute of Technology Meghalaya, Shillong 793003, India}
\cortext[mycorrespondingauthor]{Corresponding author}
\fntext[myfootnote]{The work was supported by Department of Science and Technology-Science and Engineering Research Board (grant no. ECR/2017/002116)}
\ead{manideepa.saha@nitm.ac.in}

\begin{abstract}
	In this paper, we propose and analyze iterative method based on projection techniques to solve a non-singular linear system $Ax=b$. In particular, for a given positive integer $m$, $m$-dimensional successive projection method (mD-SPM) for symmetric definite matrix $A$, is generalized for non-singular matrix $A$. Moreover, it is proved that  $m$D-SPM gives better result for large values of $m$. Numerical experiments are carried out to demonstrate the superiority of the proposed method in comparison with other schemes in the scientific literature.
	
\end{abstract}


\end{frontmatter}


\section{Introduction}
\label{intro}
Consider the linear system of equations 
\begin{equation}\label{lin_sys}
Ax=b,
\end{equation}
 where $A\in\mathbb{R}^{n,n}$, $b\in\mathbb{R}^{n}$ and $x\in\mathbb{R}^n$ is an unknown vector.
In ~\cite{Uje06}, Ujevi\'c introduced a new iterative method for solving (\ref{lin_sys}). The method was considered as a generalization of Gauss-Seidel methods. In ~\cite{JinH08}, Jing and Huang interpreted Ujevi\'c's method as one dimensional double successive projection method (1D-DSPM), whereas Gauss-Seidel method as one dimensional single successive projection method (1D-SSPM). They established a different approach and called it as two dimensional double successive projection method (2D-DSPM). In \cite{Sal09}, Salkuyeh improved this method and gave a generalization of it, thereby calling it  $m$D-SPM. In an independent work ~\cite{HouW09}, Hou and Wang  explained 2D-DSPM as two-dimensional orthogonal projection method (2D-OPM), and generalized it to three dimensional orthogonal projection method (3D-OPM). All these works address systems in which the matrix $A$ in \ref{lin_sys} is symmetric positve definite(SPD). In this paper, we generalize the $m$D-SPM and use it on any non-singular system. The proposed method is called as $m$D-OPM, where OPM refers to `orthogonal' as well as `oblique' projection method.

Given an initial approximation $x_{0}$, a typical projection method  for solving \eqref{lin_sys}, on the subspace $\mathcal{K}$ (known as the search subspace) and orthogonal to the subspace $\mathcal{L}$ (known as the constraint subspace), is to find an approximate solution $x$ of (\ref{lin_sys}) by imposing the Petrov-Galerkin condition \cite{Saad} that
\begin{equation}\label{PGcond}
x\in x_{0} +\mathcal{K}\text{   and   } b-Ax\perp \mathcal{L}
\end{equation}

In case of orthogonal project method, search space and the constraint spaces are same, whereas in oblique projection method, they are different. 
The elementary Gauss-Seidel method can be considered as an one dimensional OPM with $\mathcal{K}=\mathcal{L}=\Span\{e_{i}\}$, where $e_{i}$ is the ith column of the identity matrix. In~\cite{HouW09}, authors proposed three-dimensional OPM ($3$D-OPM) and showed both theoretically and numerically that $3$D-OPM gives better (or atleast the same) reduction of error than 2D-OPM in~\cite{JinH08}. In \cite{Sal09}, author proposed a generalization of $2$D-OPM as well as gave a way to chose the subspace $\mathcal{K}$.  We put forward the $m$-dimensional OPM ($m$D-OPM) by considering $m$-dimensional subspaces $\mathcal{K},\mathcal{L}$, where, for oblique projection we take $\mathcal{L}=A\mathcal{K}$.   At each iterative step,  $\mathcal{K},\mathcal{L}$ are taken as $m$-dimensional subspaces and each iteration is cycled for $i=1,2,\ldots,n$, until it converges.

The paper is organized as follows: In Section 2, a theoretical proof of the advantage of chosing a larger value of $m$ in $m$D-OPM is provided and also convergence of $m$D-OPM for an SPD system is shown, which supplements the work in \cite{Sal09}. Section 3 shows the application of $m$D-OPM to more general non-singular systems. Lastly, in section 4, numerical examples are considered for illustration. 

\section{mD-OPM for symmetric matrices}\label{sec2}
Throughout this section, the matrix $A$ under cosideration is assumed to be SPD. From the numerical experiments provided in \cite{Sal09}, it is observed that $mD$-OPM provides better (or at least same) result with larger value of $m$.  In this section we prove it theoretically. 

In $mD$-OPM, $\mathcal{K}$(= $\mathcal{L}$) is considered as an $m$-dimensional space. If $\mathcal{K}=\mathcal{L}=\Span\left\{v_{1},v_2,\ldots,v_{m}\right\}$, and $V_{m}=[v_{1},v_2,\ldots,v_{m}]$, a basic projection step  for an $mD$-OPM is defined in~\cite{Sal09,Saad} as :

\vspace{-5mm}
\begin{eqnarray}
&\text{Find }  x^{(i+1)}\in x^{(i)}+\mathcal{K} \text{ such that }  b-Ax^{(i+1)}\perp \mathcal{K} \label{pro_step}\\
\text{Equivalently, }&x^{(i+1)}=x^{(i)}+V_m(V_m^TAV_m)^{-1}V_m^Tr^{(i)}\text { where } r^{(i)}=b-Ax^{(i)}\label{perp}
\end{eqnarray}

If $x^*$ is the exact solution of (\ref{lin_sys}),  the quantity $\|x^{(i+1)}-x_{*}\|_{A}^{2}-\|x^{(i)}-x_{*}\|_{A}^{2}$ is defined as the error reduction at the $i$th iteration step of $m$D-OPM~\eqref{pro_step} and is denoted by $E.R_{mD}$ as considered in \cite{HouW09}. In Theorem 1 of ~\cite{Sal09}, author proved that $E.R_{mD}\leq 0$. In particular, $E.R_{mD}< 0$, if $r^{(i)}$ is not perpendicular to $\mathcal{K}(=\mathcal{L})$. To prove the main theorem of this section, we need the following Lemma.
\begin{lemma}\rm\label{f-func-eq} If $x^{(i+1)}$s are defined as in (\ref{perp}), then $E.R_{mD}=\left\langle z^{(i)},p^{(i)}\right\rangle$, where $p^{(i)}=-V_{m}^Tr^{(i)}$ and $z^{(i)}=-(V_{m}^TAV_{m})^{-1}p^{(i)}$.
\end{lemma}
\begin{proof} Proof follows from the fact $E.R_{mD}=-(V_{m}^Tr^{(i)})^T(V_{m}^TAV_{m})^{-1}V_{m}^Tr^{(i)}$ in the proof of Theorem 1 of ~\cite{Sal09}.
\end{proof}
For any positive integer $k$, write $B_k=V_{k}^TAV_{k}$. Let $l\in\{1,2,...,m-1\}$. If we choose  $\dim\mathcal{K}=l$  in \eqref{pro_step},  then by Lemma~\ref{f-func-eq}
\begin{equation}\label{erld} 
E.R_{lD}=\langle y^{(i)},\tilde{p}^{(i)}\rangle
\end{equation}
 
where $\tilde{p}^{(i)}=-V_{l}^T\tilde{r}^{(i)}$ and $y^{(i)}=-(B_l)^{-1}\tilde{p}^{(i)}$. Note that
$B_l=B_{m}[\alpha]$, $y^{(i)}=z^{(i)}[\alpha]$, $\tilde{p}^{(i)}=p^{(i)}[\alpha]$ and $\alpha = \left\lbrace 1,2,...,l\right\rbrace $.
\begin{theorem}\rm\label{error_red}
	$E.R_{mD}\leq E.R_{lD}$, when $m>l$.
\end{theorem}
\begin{proof} For simplicity, we write $y^{(i)}:=y, ~z^{(i)}:=z, p^{(i)}:=p, ~B^{(i)}:=B$.   For $\alpha \subsetneq \left\lbrace 1,2,...,m\right\rbrace,$ define $\tilde{z}\in\mathbb{R}^{n}$ as $\tilde{z}(\alpha)=y$ and $0$ elsewhere.
By Lemma~\ref{f-func-eq}, it is sufficient to show that $z^{T}p\leq y^{T}\tilde{p}$, or equivalently, $z^{T}Bz\geq \tilde{z}^{T}B\tilde{z}$.  
	Since $B$ is a positive definite matrix,  $(z-\tilde{z})^{T}B(z-\tilde{z})\geq 0$ which implies that
	\begin{align}\label{laststep}
		z^{T}Bz-\tilde{z}^{T}B\tilde{z}\geq \tilde{z}^{T}Bz+z^{T}B\tilde{z}-2\tilde{z}^{T}B\tilde{z}.
	\end{align}
	However, $\tilde{z}^{T}Bz+z^{T}B\tilde{z}-2\tilde{z}^{T}B\tilde{z}
	=-\tilde{z}^{T}p-p^{T}\tilde{z}-2y^{T}\tilde{B}{y}= -2\tilde{p}^{T}y+2y^{T}\tilde{p}=0$.
	
Thus from \eqref{laststep}, we get 
	$z^{T}Bz\geq \tilde{z}^{T}B\tilde{z}.$
\end{proof}

\begin{corollary}\rm $mD$-OPM defined in \eqref{perp} converges.
\end{corollary}
\begin{proof}
In \cite{JinH08,HouW09},  it is  shown that  $E.R_{mD}\leq 0$, for $m=2,3$, which assures the convergence of $m$D-OPM, for any $m$. Hence the conclusion follows from Theorem~\ref{error_red}.
\end{proof}

\section{m-dimensional oblique projection method for non-singular matrices}\label{sec3}

In this section we present new $m$-dimensional oblique projection method (mD-OPM) to solve nonsingular system \eqref{lin_sys}. Assume that $\dim \mathcal{K}=\dim \mathcal{L}=m$, with $m\ll n$. Take $V_{m}=[v_{1},v_2,\ldots,v_{m}]$, and $W_{m}=[w_{1},w_2,\ldots,w_{m}]$ so that columns of $V$ and $W$ form bases for $\mathcal{K}$ and $\mathcal{L}$, respectively. If $\mathcal{L}=A\mathcal{K}$, then the oblique projection iterative steps, discussed in \eqref{PGcond}, are given as follows \citep{Saad}:
\begin{equation}\label{pro_step_mat}
	x^{(i+1)}=x^{(i)}+V_{m}(W_{m}^{T}AV_{m})^{-1}W_{m}^{T}r^{(i)}.
\end{equation}
where $r^{(i)}=b-Ax^{(i)}$ is the residual in the $i$th iteration step.

Choose $\mathcal{L}=A\mathcal{K}$. Then $x^{(i+1)}$ as defined in \eqref{PGcond} minimizes the $2$-norm of the residual $r^{(i+1)}$ over $x \in x^{(i)}+\mathcal{K}$ (see, Ch.5~in~\cite{Saad}). Throughout this section, $\|.\|$ represents $2$-norm in the Euclidean space $\mathbb{R}^n$ and we drop the suffix $m$ which signifies the dimension of $V_{m}$ and $W_{m}$.

As $\mathcal{L}=A\mathcal{K}$, we may take $W=AV$. Then \eqref{pro_step_mat} reduces to 
\begin{equation}\label{pro_step_mat1}
x^{(i+1)}=x^{(i)}+VW^{\dagger}r^{(i)},
\end{equation} 
where $W^{\dagger}$ denotes the pseudo-inverse of $W$ so that $r^{(i+1)}=b-Ax^{(i+1)}=r^{(i)}-WW^{\dagger}r^{(i)}.$ Main goal of this section is to prove the convergence of \eqref{pro_step_mat1}. Following lemma will help to reach our goal.

\begin{lemma}\rm\label{lem_conv_obl} If $\sigma_{1}$ is the maximum singular valur of $A$, and $y=W^{T}r^{(i)}$, then
\[\|r^{(i)}\|^{2}-\|r^{(i+1)}\|^{2}\geq \frac{1}{\sigma_1^2}\|y\|^2 .\]
\end{lemma}

\begin{proof} As $(WW^{\dagger})^{T}=WW^{\dagger}$ and $WW^{\dagger}W=W$,
we have,
\begin{equation}\label{er}
\|r^{(i+1)}\|^{2}=(r^{(i)^{T}}-r^{(i)^{T}}(WW^{\dagger })^{T})(r^{(i)}-WW^{\dagger }r^{(i)})=\|r^{(i)}\|^{2}-r^{(i)^{T}}WW^{\dagger }r^{(i)}
\end{equation}
Using Courant-Fisher min-max principle\cite{Saad}, from \eqref{er} we achieve,
\begin{eqnarray}\label{Wsing}
\|r^{(i)}\|^{2}-\|r^{(i+1)}\|^{2}&=y^{T}(W^{T}W)^{-1}y&=\dfrac{\langle (W^{T}W)^{-1}y,y\rangle }{\|y\|^{2}}\|y\|^{2}\nonumber\\
&&\geq \lambda _{\min}(W^{T}W)^{-1}\|y\|^{2}\nonumber\\
&&=\dfrac{1}{\lambda _{\max}(W^{T}W)^{-1}}\|y\|^{2}\nonumber\\
&&=\dfrac{1}{(\sigma _{\max}(W))^{2}}\|y\|^{2},
\end{eqnarray}
where $\lambda _{\min}$, $\lambda _{\max}$ denote the  maximum and minimum eigenvalues, and $\sigma _{\max}$, $\sigma_{\max}$  denotes the maximum and minimum singular values of the corresponding matrix, respectively.

Let $W=\tilde{U}\sum \tilde{V}^{T}$ be the singular value decomposition of $W$. If $\tilde{U}=[\tilde{u_{1}},\tilde{u_{2}}, \dots, \tilde{u_{n}}]$, and $\tilde{V}=[\tilde{v_{1}},\tilde{v_{2}}, \dots, \tilde{v_{m}}]$, then $W\tilde{v_{1}}=\sigma _{1}(W)\tilde{u_{1}}$ so that
$$(\sigma _{1}(W))^{2}=\|\sigma _{1}(W)\tilde{u_{1}}\|^{2}=\|W\tilde{v_{1}}\|^{2}=\|AV\tilde{v_{1}}\|^{2}\leq \|AV\|^{2}\leq \|A\|^{2}\|V\|^{2} =\sigma _{1}^{2}.$$
Hence the result follows from \eqref{Wsing}.
\end{proof}

In Theorem 3 of~\cite{Sal09}, author  provided the convergence of the method \eqref{perp} for SPD matrices, and also gives an idea to choose the optimal vectors $v_{i}$. Similar ideas is used to prove the convergence of \eqref{pro_step_mat1}. Next theorem is due to \cite{HorJ08} (see Ch 3, Cor 3.1.1), which gives the relation between singular values of a matrix and its submatrices.
\begin{theorem}\rm\cite{HorJ08}\label{sing_sub}~If $A$ is an $m\times n$ matrix and $A_{l}$ denotes a submatrix of $A$ obtained by deleting a total of $l$ rows and/or columns from $A$, then
\[\sigma_{k}(A)\geq \sigma_k (A_l)\geq \sigma_{k+l}(A),~~~~k=1:\min\{m,n\}\]
where the singular values $\sigma_{i}$'s are arranged decreasingly.
\end{theorem}

We now prove our main theorem, in which the singular values of matrix under consideration, are assumed to be arranged in decreasing order.

\begin{theorem}\rm\label{conv_obl} Let $\sigma_1\geq\sigma_2\geq \ldots\geq \sigma_n$ be the singular values of $A$. If $i_1<i_2<\ldots<i_m$, and $v_{j}=e_{i_{j}}$, $i_{j}$th column of the identity matrix, then
\begin{equation}\label{eqn_obl-pro}
\|r^{(i+1)}\|^{2}\leq \left( 1-\dfrac{\sigma_{n}^{2}}{\sigma _{1}^{2}}\right) \|r^{(i)}\|^{2}.
\end{equation}
\end{theorem}

\begin{proof} Let $\alpha=\{i_1,i_2,\ldots,i_m\}$, and $A_{m}=A[:,\alpha]^T$. Then $y=W^{T}r^{(i)}=A_{m}r^{(i)}$. Since $A_{m}$ has full row rank, as shown in \eqref{Wsing}, we can infer
\[\|y\|^{2}=\dfrac{\|A_{m}r^{(i)}\|^{2}}{\|r^{(i)}\|^{2}}\|r^{(i)}\|^{2}\geq \sigma_{min}^{2}(A_{m})\|r^{(i)}\|^{2}=\sigma_{m}^{2}(A_{m})\|r^{(i)}\|^{2}.\]
Taking $l=n-m$ and $k=m$ in Theorem~\ref{sing_sub}, we get
$\|y\|^{2}\geq \sigma_{n}^{2}\|r^{(i)}\|^{2}.$
So, from Lemma~\ref{lem_conv_obl} we conclude that
$$\|r^{(i)}\|^{2}-\|r^{(i+1)}\|^{2} \geq \dfrac{\sigma_{n}^{2}}{\sigma _{1}^{2}}\|r^{(i)}\|^{2},$$
Hence the conclusion follows.
\end{proof}

\begin{remark}\rm
	The quantity $\|r^{(i)}\|^{2}-\|r^{(i+1)}\|^{2}$ is greater when larger value of $m$ is chosen. This can be seen from ~\eqref{er} and following similar steps used in proving Theorem~\ref{error_red}.
\end{remark}

\begin{remark}\rm Under the assumption in Theorem~\ref{conv_obl}, equation \eqref{eqn_obl-pro} suggests that  the iteration process in (\ref{pro_step_mat1}) converges.
\end{remark}

\section{Numerical Experiments} 
In this section comparison of $m$D-OPM is established with various methods, like, CGNR, GMRES and Craig's method \cite{Craig} for any non-singular linear system. 

The algorithm of the $mD$-OPM, discussed in Section~\ref{sec3}, is as follows, which is same as proposed in~\cite{Sal09} by considering the symmetric system $A^TA=A^Tb$.

\begin{algorithm}
	\caption{\cite{Sal09} A particular implementation for arbitrary dimensional OPM} 
	1. Chose an initial guess $x^{(0)}\in \textbf{R}^{n}$ and decide on $m$, the number of times each component of  $x^{(0)} $ is improved in each iteration.\\
	2. Until Convergence, Do\\
	3. $x=x^{(0)}$.\\
	4. For $i=1,2,\dots ,n$, Do\\
	5.Select the indices $i_{1},i_{2},\dots ,i_{m}$ of $r$\\
	6. $E_{m}=\left[ e_{i_{1}},e_{i_{2}},\dots ,e_{i_{m}}\right] $\\
	7.Solve $(E_{m}^{T}A^{T}AE_{m})y_{m}=E_{m}^{T}A^{T}r$ for $y_{m}$\\
	8.$x=x+E_{m}y_{m}$\\
	9.$r=r-AE_{m}y_{m}$\\
	10.End Do\\
	11.$x^{(0)}=x$\\
	12.End Do.
	
\end{algorithm}

The experiments are done on a PC-Intel(R) Core(TM) i3-7100U CPU @ 2.40 GHz, 4 GB RAM. The computations are implemented in MATLAB 9.2.0.538062. The initial guess is $x^{(0)}=[0,0,\dots ,0]^{T}$ and the stopping criteria is $\|x^{(i+1)}-x^{(i)}\|<10^{-12}$.

While doing comparisions with $m$D-OPM, we consider different values of $m$ to get various results. The theory suggests that $m$D-OPM will have a good convergence for matrices whose singular values are closely spaced. Hence we chose the matrices accordingly.  

\begin{example}\rm\label{ex1}
	The first matrix is a symmetric $n \times n$ Hankel matrix with elements $A(i,j)=\frac{0.5}{n-i-j+1.5}$. The eigen values of $A$ cluster around $-\frac{\pi }{2}$ and $\frac{\pi }{2}$ and the condition number is of $O(1)$. The matrix is of size $100$. Comparision is done for different values of $m$ as well as with the CGNR, GMRES and Craig's method.
\end{example}
\begin{table}
	\caption{Results for \ref{ex1}}
	\begin{center}
		\begin{tabular}{|c||c|c|}
			\hline
			Iteration Process & No of Iterations & Residual   \\ \hline
			$6$D-OPM & 14 & $3.5755 \times 10^{-12}$\\ \hline
			$10$D-OPM & 8 & $4.6142 \times 10^{-12}$\\ \hline
			$50$D-OPM & 2 & $3.8 \times 10^{-15}$\\ \hline
			GMRES & 10 & $3.7\times 10^{-15}$\\ \hline
			CGNR & 9 & $5.3427\times 10^{-15}$\\ \hline
			Craig & 9 & $4.9704\times 10^{-15}$\\ \hline 
		\end{tabular}
	\end{center}
\end{table}
\newpage

\begin{example}\rm\label{ex2}
	We consider a square matrix of size $n$ with singular values $1+10^{-i}$, $i=1:n$. This is again a matrix with extremely good condition number. For such a well-conditioned matrix, $m$d-dspm works like a charm and is better than the CGNR. The matrix taken here is of size $400$. 
\end{example}

\begin{table}[h]
	\caption{Results for \ref{ex2}}
	\begin{center}
		\begin{tabular}{|c||c|c|}
			\hline
			Iteration Process & No of Iterations & Residual   \\ \hline
			$4$D-OPM & 1 & $2.1618 \times 10^{-15}$\\ \hline
			
			CGNR & 6 & $5.3328\times 10^{-15}$\\ \hline
			Craig & 6 & $5.4702\times 10^{-15}$\\ \hline
		\end{tabular}
	\end{center}
\end{table}

\section{Conclusion}
$m$D-OPM, presented in this paper, is a generalization of $m$D-SPM \cite{Sal09}, and can be applied to any non-singular system. Numerical experiments showed that this method is at par with other established methods. The way in which the search subspace is chosen put this method at a clear advantage over GMRES, because in GMRES, the orthogonalisation through Arnoldi process can lead to infeasible growth in storage requirements.

\section*{Reference}

\bibliography{mybibfile}

\end{document}